\documentclass{amsart}
\usepackage{amsmath,amsfonts,amssymb,amsthm,hyperref}

\usepackage{color, tikz,enumitem}
\usepackage{bm}

\numberwithin{equation}{section}
\newtheorem{claim}{claim}[section]
\newtheorem{theorem}[claim]{Theorem}

\newtheorem{lemma}[claim]{Lemma}
\newtheorem{proposition}[claim]{Proposition}

\newtheorem{corollary}[claim]{Corollary}

\newtheorem{remark}[claim]{Remark}

\let\infp\relax \DeclareMathOperator*\infp{\vphantom{p}inf}

\newcommand{\ds}{\Omega_n}
\newcommand{\p}[1][p]{{\mathcal{S}_{#1}}}
\newcommand{\Pn}{\mathcal{P}_n}
\newcommand{\tr}{\operatorname{tr}}

\author{Ludovick Bouthat, Javad Mashreghi and Frédéric Morneau-Guérin}

  \title{On the Geometry of the Birkhoff Polytope \\ {II. The Schatten $\bm{p}$-norms}}

\begin{document}

\begin{abstract}
{In the first of this series of two articles, we studied some geometrical aspects of the Birkhoff polytope, the compact convex set of all $n \times n$ doubly stochastic matrices, namely the Chebyshev center, and the Chebyshev radius of the Birkhoff polytope associated with metrics induced by the operator norms from $\ell_n^p$ to $\ell_n^p$ for $1 \leq p \leq \infty$. In the present paper, we take another look at those very questions, but for a different family of matrix norms, namely the Schatten $p$-norms, for $1 \leq p < \infty$. While studying these properties, the intrinsic connection to the minimal trace, which naturally appears in the assignment problem, is also established.
}
\end{abstract}


\maketitle

%%%%%%%%%%%%%%%%%%%%%%%%%%%%%%%%%%%%%%%%%%%%%%%%%%%%%%%%%%%%%%%
%%%%%%%%%%%%%%%%%%%%%%%%%%%%%%%%%%%%%%%%%%%%%%%%%%%%%%%%%%%%%%%
%%%%%%%%%%%%%%%%%%%%%%%%%%%%%%%%%%%%%%%%%%%%%%%%%%%%%%%%%%%%%%%
%%%%%%%%%%%%%%%%%%%%%%%%%%%%%%%%%%%%%%%%%%%%%%%%%%%%%%%%%%%%%%%

\section{Introduction}

% {\color{purple}
% A square matrix $D= [d_{ij}]$ is said to be \textit{doubly stochastic} if every entry of $D$ is non-negative and if each row and each column of $D$ sums up to 1, i.e.,
% $$
% d_{ij} \geq 0 ~~\quad\&~~\quad \sum\limits_{i=1}^n d_{ij}= \sum\limits_{j=1}^n d_{ij}=1,
% $$
% for all $i, j = 1,2, \dots, n$. Doubly stochastic matrices appear naturally in many different mathematical contexts such as, for instance, in the theory of majorization \cite{Marshall2011}, and in the \emph{assignment problem} \cite{BURKARD2002257} (more on this in Section \ref{subsec - assign}).

% Let $\ds$ denote the set of $n\times n$ doubly stochastic matrices. It is well known that $\ds$ is a semigroup with respect to matrix multiplication and that it is a convex polytope (i.e., a compact convex set with a finite number of extreme points) in the Euclidean space of dimension $n^2$. Moreover, it was shown by Birkhoff \cite{Birkhoff1946} that the extreme points of $\ds$ are precisely the $n\times n$ permutation matrices. We write $\ds=\conv(\Pn)$, where  $\Pn$ denotes the set of $n\times n$ permutation matrices and $\conv(\cdot)$ designates the convex hull of the set that is operated on. More specifically, each $D\in\ds$ admits a (not necessarily unique) \emph{Birkhoff decomposition} $D=\sum_{i=1}^r \alpha_i P_i$, where $P_i\in\Pn$, $\alpha_i \geq 0$, and $\sum_{i=1}^r \alpha_i=1$. Due to this characterization, $\ds$ is sometimes referred to as the \emph{Birkhoff polytope}. }

A square matrix is said to be \textit{doubly stochastic} if its entries are nonnegative and all its row and column sums are one. It is readily verified that the set of $n \times n$ doubly stochastic matrices is closed under matrix multiplication.

Note that a doubly stochastic matrix can be interpreted as the \textit{adjacency matrix} of a \textit{weighted directed graph} such that at each vertex the sum of the weights of the incoming edges as well as the sum of the weights of the outgoing edges are equal to one \cite{gharesifard2010does}. Through this association, doubly stochastic matrices play a key role in networked control problems (NSC) \cite{bullo2009distributed, ren2008distributed}.

The connection between such matrices and graph theory in particular, and combinatorics in general does not stop there. The celebrated Birkhoff's theorem asserts that every $n \times n$ doubly stochastic matrix can be written (though not necessarily in a unique way) as a convex combination of $n \times n$ permutation matrices, and the $n \times n$  permutation matrices are precisely the extreme points of $\Omega_n$, the set of $n \times n$ doubly stochastic matrices often referred to as the \textit{Birkhoff polytope}. This foundational result turns out to be logically equivalent to a long list of remarkably powerful theorem in combinatorics which includes the K\H{o}nig--Egerv\'ary theorem, K\H{o}nig's theorem, Menger's theorem, Dilworth's theorem as well as Ford and Fulkerson's \textit{max-flow min-cut} theorem \cite{reichmeider1984equivalence}.

%The geometry of the Birkhoff polytope has been an active subject of research for more than half a century. For instance, in 1977, in a series of four papers, Brualdi and Gibson \cite{BrualdiGibson4,BrualdiGibson1,BrualdiGibson2,BrualdiGibson3} studied the Euclidean geometry structure of $\ds$. In particular, they described the faces, the edges and the facets of $\ds$. In 1996, Billera and Sarangarajan \cite{BilleraSarangarajan1996} pursued this line of study, while also considering two other related polytopes. Then, from 1999 up to 2016, several articles studying the volume of $\ds$ were published. In particular, formulas for the volume of $\ds$ were given by Sturmfels \cite{Sturmfels1997} for $n\leq7$,  by Chan and Robbins \cite{ChanRobbins1999} for $n=8$, and by Beck and Pixton \cite{BeckPixton2003} for $n=9,10$. As for the case of $10\leq n \leq 15$, various estimates were obtained in 2014 by Emiris and Fisikopoulos \cite{EmirisFisikopoulos2014}, and in 2016 by Cousins and Vempala. \cite{CousinsVempala2016}. Meanwhile, De Loera, Liu and Yoshida \cite{DeLoeraLiuYoshida2009} provided an explicit combinatorial formula for the volume of $\ds$ in 2009.

In \cite{part1}, the geometry of the Birkhoff polytope is studied. More specifically, the Chebyshev center and the Chebyshev radius of $\ds$ relative to the operator norms from $\ell^p_n$ to $\ell^p_n$ ($1 \leq p \leq \infty$) is determined. These concepts are interesting from a theoretical point of view. However, they also have received attention in some practical settings, notably in two papers from 1998 by Glunt, Hayden and Reams \cite{Glunt1998} and by Khoury \cite{Khoury}. In the former, while studying numerical simulation of large linear semiconductor circuit networks, the authors naturally came across the following question: \textit{Given a matrix $B$ of order $n$ subject to the constraints $e_1^\intercal D^k e_1 = e_1^\intercal B^k e_1$, $k\geq1$, where $e_1 = (1,0,\dots,0)^\intercal$, which generalized doubly stochastic matrix $D$ is \emph{closest} to $B$ in the sense of Frobenius?} They were able to give an algorithm to numerically find the solution. In the latter, Khoury \cite{Khoury} independently studied the same question, bar the constraints mentioned above. He showed that $D = WBW+J_n$, where $W=I_n-J_n$ and $J_n$ is the $n\times n$ matrix with every entry uniformly equal to $1/n$. This line of investigation was later pursued by Glunt, Hayden and Reams in \cite{Glunt2008} and further developed in the more specific case of doubly stochastic matrices (as opposed to generalized doubly stochastic matrices) in \cite{MR2306262} by Bai, Chu and Tan in 2007.

The present article continues the line of research pursued in \cite{part1}. Therein, in the process of determining the Chebyshev center and radius of $\ds$ relative to the operator norm from $\ell^2_n$ to $\ell^2_n$, the case of the Frobenius norm is also addressed. A natural generalization of this norm is the \emph{Schatten $p$-norms}. Hence, in this paper, we propose to also determine the Chebyshev center and the Chebyshev radius of $\ds$, but this time relative to the Schatten $p$-norms ($1\leq p < \infty$). The case of the Schatten $\infty$-norm is omitted since it coincide with the operator norm from $\ell^2_n$ to $\ell^2_n$, which is already covered in \cite{part1}. In this paper, we also address the minimal bounding ball's radius for the Birkhoff polytope and the smallest enclosing ball problem.

In Section \ref{sec - def}, we establish some preliminary results. More precisely, we recall some basic definitions in Section \ref{subsec - def} and Section \ref{subsec - prop}, we present a few elementary properties of doubly stochastic matrices in Section \ref{sec - prop}, we state some basic facts about the minimal trace of a matrix in \ref{subsec - assign}, and we recall some general results on the Chebyshev radius proven in \cite{part1}.  %\textcolor{red}{In Section \ref{Sec: Remarks}, we give a summary of the main results contained in the present paper.} 
In Section \ref{sec - size}, we determine the minimum and maximum distance of an element of the Birkhoff polytope from the origin. In Section \ref{sec - ball}, we study the minimal bounding ball of the Birkhoff polytope. Among other things, we find an explicit formula for the radius of the smallest enclosing ball of $\ds$ centered at $D$ in the case of the Frobenius norm, for any doubly stochastic matrix $D$. In doing so, interesting connections with the assignment problem are highlighted. Finally, in Section \ref{sec - Chebyshev}, we study the Chebyshev center and the Chebyshev radius of $\ds$ relative to the Schatten $p$-norms.

%%%%%%%%%%%%%%%%%%%%%%%%%%%%%%%%%%%%%%%%%%%%%%%%%%%%%%%%%%%%%%%
%%%%%%%%%%%%%%%%%%%%%%%%%%%%%%%%%%%%%%%%%%%%%%%%%%%%%%%%%%%%%%%
%%%%%%%%%%%%%%%%%%%%%%%%%%%%%%%%%%%%%%%%%%%%%%%%%%%%%%%%%%%%%%%
%%%%%%%%%%%%%%%%%%%%%%%%%%%%%%%%%%%%%%%%%%%%%%%%%%%%%%%%%%%%%%%

\section{Definitions, Properties and Preliminary results} \label{sec - def}

%%%%%%%%%%%%%%%%%%%%%%%%%%%%%%%%%%%%%%%%%%%%%%%%%%%%%%%%%%%%%%%

\subsection{Basic definitions}
\label{subsec - def}

In this section, we recall the definition of some basic notions which are essential in this paper. Secondly, we state some well-known properties of doubly stochastic matrices. Thirdly, we define the minimal trace of a $n \times n$ matrix and discuss some of its elementary properties. Finally, we recall some known results about the Chebyshev center and radius of $\ds$ relative to a metric induced by a permutation-invariant norm.

%%%%%%%%%%%%%%%%%%%%
% \subsubsection*{Vector $p$-norms}

% Recall that, for $p \geq 1$, the \textit{$p$-norm} of a given vector ${x = (x_1, \dots, x_n)}$ is defined by
% \[
% \|x\|_p \,=\, \left( \sum_{k=1}^{n} |x_k|^p \right)^{\!\!1/p},
% \]
% and the $\infty$-norm by
% \[
% \|x\|_\infty \,=\, \max\{ |x_1|, \dots, |x_n|\}.
% \]
% When $\mathbb{C}^n$ is equipped with the vector $p$-norm, it is customary to denote it by $\ell^p_n(\mathbb{C})$, or by $\ell^p_n$ for short. 

% %%%%%%%%%%%%%%%%%%%%
% \subsubsection*{Operator norms induced by the vector $p$-norms}

% Any $n \times n$ matrix $A$ can be interpreted as an operator from $\ell^p_n$ to $\ell^p_n$, in which case its operator norm is given by
% \[
% \|A\|_{\lp} \,:=\, \sup_{x\neq 0}{\frac {\|Ax\|_p}{\|x\|_p}}.
% \]
% %
% We recall that for all $1\leq p\leq \infty$, we have $\|B\|_{\lp} = \|B^*\|_{\lp[q]}$, where $*$ denotes the Hermitian conjugate and $q$ is the H\"older conjugate of $p$, i.e.,  $\tfrac{1}{p} + \tfrac{1}{q} = 1$  \cite[p.\,357]{HornJohnson2013}.

% Since the sets $\ds$ and $\Pn$ are invariant under the conjugation action, many operators considered below have the same norm, whether we see them as $\ell^p_n \to \ell^p_n$ or as $\ell^q_n \to \ell^q_n$ mappings. Consequently, when such situation arises, for the sake of brevity we shall only consider the case $1\leq p \leq 2$.

%%%%%%%%%%%%%%%%%%%%
\subsubsection*{Singular values}

Given an $m \times n$ matrix $A$, the $i$-th singular value of $A$ is given by
$$ 
\sigma_i(A) \,= \!\!\!\!\infp\limits_{\substack{ V \subseteq \mathbb{C}^n \\ \dim(V) = n-i+1}}\sup\limits_{\substack{x \in V \\ \|x\|_2 = 1}} \!\|Ax\|_2,\quad\quad (i = 1, 2, \dots, \min\{m,n\}).
$$
It is easy to see that
$$
\sigma_1(A) \,\geq\, \sigma_2(A) \,\geq\, \cdots \,\geq\, \sigma_{\min\{m,n\}}(A) \,\geq\, 0.
$$
Moreover, the singular values of $A$ are also given by the square root of the nonnegative eigenvalues of $A^*A$, where $A^*$ denotes the conjugate transpose of $A$. In this case, the singular values are not necessarily ordered and thus we assume without loss of generality that they are.

%%%%%%%%%%%%%%%%%%%%
\subsubsection*{Schatten norms}

For $p \geq 1$, the Schatten $p$-norm of a $m \times n$ matrix $A$, denoted by $\|A\|_{\p}$, is defined as the $p$-norm of its singular values, that is
\[
\|A\|_{\mathcal{S}_p} \,:=\, \Bigg( \sum_{i=1}^{\min\{m,n\}} \!\!\sigma_i^p(A)\Bigg)^{\!1/p} \!\!
\]

The case $p=2$ yields the well-known \emph{Frobenius norm}, which admits an easier characterization. Indeed, this latter matrix norm is induced by the inner product $\langle A,B\rangle_{\text{F}} := \tr(A B^*)$, where $A$ and $B$ belong to $M_{m,n}(\mathbb{C})$. Hence, the Schatten $2$-norm of $A=[a_{ij}]_{m \times n}$ is also given by
\[
\|A\|_{\mathcal{S}_2} \,=\, \|A\|_{\text{F}} \,=\, \left( \sum_{i=1}^{m}\sum_{j=1}^{n} |a_{ij}|^2 \right)^{\!\!1/2} \!\!.
\]

%%%%%%%%%%%%%%%%%%%%%%%%%%%%%%%%%%%%%%%%%%%%%%%%%%%%%%%%%%%%%%%

\subsection{Useful properties of the Schatten \texorpdfstring{$p$}{p}-norms}
\label{subsec - prop}

In what follows, we present a few well-known features of the Schatten $p$-norms that we shall use later.

%%%%%%%%%%%%%%%%%%%%
\subsubsection*{Monotonicity}

For $1 \leq p \leq q$,
\[
\|A\|_{\mathcal{S}_1} \,\geq\, \|A\|_{\mathcal{S}_p} \,\geq\, \|A\|_{\mathcal{S}_q}.
\]
This is an obvious consequence of well-known results on inclusions between $\ell^p$ spaces.

%%%%%%%%%%%%%%%%%%%%
\subsubsection*{Permutation-invariance}\label{PINdef}

Recall that any complex $m \times n$ matrix $A$ admits a \textit{singular value decomposition}, that is a factorization of the form
\[
A \,=\, U \Sigma V^*,
\]
where $U$ is an $m \times m$ complex unitary matrix, $\Sigma$ is an $m \times n$ rectangular diagonal matrix with the singular values of $A$ on the diagonal, and $V^\ast$ is the conjugate transpose of a $n \times n$ complex unitary matrix $V$. For all $p \geq 1$, it is easy to verify that
\[
\|A\|_{\mathcal{S}_p} \,=\, \|\Sigma\|_{\mathcal{S}_p}.
\]
Similarly, the Schatten $p$-norm of any matrix having the same singular values as $A$ is equal to $\|\Sigma\|_{\mathcal{S}_p}$. This means that the Schatten $p$-norms are unitarily-invariant and, \textit{a fortiori}, permutation-invariant.

%%%%%%%%%%%%%%%%%%%%
\subsubsection*{Submultiplicativity}

For all $p \geq 1$, the Schatten $p$-norm is \textit{submultiplicative}, that is
\[
\|AB\|_{\mathcal{S}_p} \,\leq\, \|A\|_{\mathcal{S}_p} \|B\|_{\mathcal{S}_p} 
\]
for all $p \geq 1$ and for every $m \times n$ matrix $A$ and every $n \times k$ matrix $B$.

To demonstrate this elementary fact, it suffices to show that it holds true for rectangular diagonal matrices $\Sigma_A$ and $\Sigma_B$ having for diagonal entries the singular values as $A$ and $B$, respectively, in decreasing order of modulus. Under these assumptions, we have
\begin{align*}
\|\Sigma_{AB}\|_{\mathcal{S}_p}\,&=\, \|\Sigma_A\Sigma_B\|_{\mathcal{S}_p} \\&=\,  \Bigg( \sum_{i=1}^n \!\! \big(\sigma_i(A)\sigma_i(B)\big)^p\Bigg)^{\!1/p} \\&\leq\,  \sigma_1(A) \cdot \Bigg( \sum_{i=1}^n \!\!\sigma_i(B)^p\Bigg)^{\!1/p} \\&\leq\,  \Bigg( \sum_{i=1}^n \!\!\sigma_i(A)^p\Bigg)^{\!1/p} \cdot \Bigg( \sum_{i=1}^n \!\!\sigma_i(B)^p\Bigg)^{\!1/p} \\&=\,  \|\Sigma_{A}\|_{\mathcal{S}_p} \|\Sigma_{B}\|_{\mathcal{S}_p}.
\end{align*}

%%%%%%%%%%%%%%%%%%%%%%%%%%%%%%%%%%%%%%%%%%%%%%%%%%%%%%%%%%%%%%%

\subsection{Elementary spectral properties of doubly stochastic matrices}
\label{sec - prop}

As we mentioned in the first part of this series of articles, the spectrum of a doubly stochastic matrix $D$ is bounded by 1, and it always includes the scalar 1 (associated to the eigenvector $e=(1,1,\dots,1)^\intercal$). As for the other eigenvalues, they have modulus 1 if and only if $D$ is a permutation matrix \cite[Theorem 5]{Perfect1965}.

Let us state some other useful results relating to doubly stochastic matrices which will be put to use in our studies.
\begin{enumerate}[label=(\roman*)] 
    \item\label{prop - conv-max} A convex real-valued function on $\Omega_n$ attains its maximum at a permutation matrix   \cite[Corollary 8.7.4]{HornJohnson2013}.
    \item $DJ_n = J_n D = J_n$ for every $n \times n$ doubly stochastic matrix $D$, where $J_n$ is the $n \times n$ matrix where every entry is equal to $1/n$. 
    \item\label{lem - bonus3} $J_n$ is the uniform convex combination of all the $n \times n$ permutation matrices, i.e., $J_n = \frac{1}{n!} \sum_{P\in\Pn} \!P$ \cite{part1}.
\end{enumerate}

%%%%%%%%%%%%%%%%%%%%%%%%%%%%%%%%%%%%%%%%%%%%%%%%%%%%%%%%%%%%%%%

%\subsection{A special doubly stochastic matrix}

% {\color{red}
% The $n \times n$ matrix where every entry is equal to $1/n$, which we denoted by $J_n$, plays an important role in the whole theory.  This doubly stochastic matrix is special in a number of regards. Firstly, it acts as \textit{the absorbing element} for $\ds$. That is to say $DJ_n = J_n D = J_n$ for every $n \times n$ doubly stochastic matrix $D$. Secondly, as expressed by the following lemma, it is the \textit{isobarycenter} of $\Pn$.

% \begin{lemma}\cite{part1}\label{lem - bonus3}
% The matrix $J_n$ is the uniform convex combination of all the $n \times n$ permutation matrices, i.e.,
% \[
% J_n\,=\,\frac{1}{n!} \sum_{P\in\Pn} P.
% \]
% \end{lemma}
% }

% \begin{proof}
% Let $D$ be the doubly stochastic matrix given by the convex combination $\frac{1}{n!} \sum_{P\in \Pn} \!P$. Since the permutation group $\Pn$ is invariant under permutations on both sides (viz. $Q\Pn R=\Pn$ for all $Q$ and $R\in\Pn$),  for all $Q, R \in \Pn$, we have
% \[
% QDR \,=\, \frac{1}{n!} \sum_{P\in\Pn} QPR \,=\, \frac{1}{n!} \sum_{S\in\Pn} S \,=\, D.
% \]
% This means that $D$ is invariant with respect to the permutation of its rows and columns.  Hence, all rows (resp. columns) of $D$ are identical. Since $D$ is doubly stochastic, we conclude every entry of $D$ is equal to $1/n$.
% \end{proof}

%%%%%%%%%%%%%%%%%%%%%%%%%%%%%%%%%%%%%%%%%%%%%%%%%%%%%%%%%%%%%%%

\subsection{The minimal trace of a matrix}
\label{subsec - assign}

The {\em minimal trace} of an $n \times n$ matrix $A$ is
\[
\tr_{\min}(A)\,:=\, \min_{P\in \Pn} \tr(AP).
\]
Alternatively, it can be defined as the minimal diagonal sum of $A$, i.e., 
\[
\tr_{\min}(A) \,= \min_{\sigma \in \text{Sym}(n)} \sum\limits_{i=1}^n a_{i\sigma(i)},
\]
where $\text{Sym}(n)$ denotes the full symmetric group of degree $n$, that is the group whose elements are the bijections from the set $\{1, 2, \dots, n\}$ onto itself.

The function $\tr_{\min} : M_n(\mathbb{R}) \rightarrow \mathbb{R}$, as well as its restriction to $\ds$, were extensively studied by Wang \cite{Wang1974}. It is also closely related to the \emph{assignment problem}, which is a fundamental combinatorial optimization problem \cite{Burkard2012}. Indeed, finding the value of $\tr_{\min}(A)$ is equivalent to solving an assignment problem associated with the matrix $A$. This, in turn, can be done with the $O(n^3)$ time algorithm known as the \emph{Hungarian algorithm} \cite{MR0266680}.

We now state a result, due to Wang \cite[Proposition 1.1]{Wang1974}, which provides interesting upper and lower bounds for the minimal trace of a doubly stochastic matrix. We also provide a complete and new elementary proof.

\begin{lemma}\label{borne_m(D)}
The minimal trace of a doubly stochastic matrix lies on the closed unit interval, i.e., $0 \leqslant \tr_{\min}(D) \leqslant 1$ for all $D \in \ds$. Moreover, both these bounds are sharp. Indeed, 
\begin{enumerate}[label=(\roman*)]
\item $\tr_{\min}(D)=0$ if and only if $D$ has a diagonal consisting entirely of zeroes;
\item $\tr_{\min}(D)=1$ if and only if $D=J_n$.
\end{enumerate}
\end{lemma}

\begin{proof}
The lower bound is an immediate consequence of the non-negativity of all the entries of $D$. As for the necessary and sufficient condition for equality to hold, it is trivial.

Let us now turn our attention towards the upper bound. Let $P$ be a permutation matrix for which the minimal trace of $D$ is realized and let $T_{ij}$ be the permutation matrix which swap the $i$-th and $j$-th columns with respect to \emph{right} multiplication. Observe that if $DP=[a_{ij}]$, the minimality of $P$ implies that
\[
\sum_{k=1}^n a_{kk} \,=\, \tr(DP) \,\leq\, \tr(DP T_{ij}) \,=\, a_{ij}+a_{ji} + \sum_{\smash{\underset{k\neq i,j}{k=1}}}^n a_{kk}
\]
for each $1\leq i, j \leq n$, $i\neq j$. It follows that
\[
a_{ii}+a_{jj} \,=\, \sum_{k=1}^n a_{kk}-\sum_{\smash{\underset{k\neq i,j}{k=1}}}^n a_{kk} \,\leq\,  a_{ij}+a_{ji},
\]
which trivially holds even if $i = j$. By summing over all $i$ and $j$ we get, on the one hand,
\[
\sum_{i,j=1}^n (a_{ii}+a_{jj}) \,=\, n \sum_{i=1}^n a_{ii} + n\sum_{j=1}^n a_{jj} \,=\, 2n\tr(DP) \,=\, 2n\tr_{\min}(D),
\]
and, on the other hand,
$$\sum_{i,j=1}^n (a_{ij}+a_{ji}) \,=\, 2\sum_{i,j=1}^n a_{ij} \,=\, 2n.$$
Therefore,
\begin{equation}\label{eq - 2nm(D)=2n}
2n\tr_{\min}(D) \,=\, \sum_{i,j=1}^n (a_{ii}+a_{jj}) \,\leq\, \sum_{i,j=1}^n (a_{ij}+a_{ji}) \,=\, 2n.
\end{equation}
Hence $\tr_{\min}(D)\leq 1$.

It only remains to discuss the conditions under which this upper bound is attained. One can easily check, by direct verification, that $\tr_{\min}(J_n)=1$. To complete the proof, suppose that $D$ is a doubly stochastic matrix $D$ such that $\tr_{\min}(D)=1$ and assume that the minimal trace of $D$ is realized by $P\in \Pn$. Set $DP=[a_{ij}]$. Under our assumption that $\tr_{\min}(D)=1$, equality must hold in \eqref{eq - 2nm(D)=2n}. But this happens if and only if 
\begin{equation}\label{E:ii=jj}
a_{ii}+a_{jj} \,=\, a_{ij}+ a_{ji}
\end{equation}
for all $1\leq i,j \leq n$. Summing up over all $i$, we obtain 
\[
\sum_{i=1}^n (a_{ii}+a_{jj}) \,=\, \tr_{\min}(D) + na_{jj} \,=\, 1 + na_{jj}.
\]
on the left-hand side, and
\[
\sum_{i=1}^n (a_{ij}+a_{ji}) \,=\, 1+1 \,=\, 2
\]
on the right-hand side. Hence $a_{jj}=1/n$ for all $1\leq j \leq n$. 

Now, let $\sigma \in \text{Sym}(n)$ be an arbitrary permutation of $\{1,2,\dots,n\}$. Then, by \eqref{E:ii=jj}, 
\[
2\tr_{\min}(D) = \sum_{i=1}^n (a_{ii}+a_{\sigma_i\sigma_i}) = \sum_{i=1}^n(a_{i\sigma_{i}}+a_{\sigma_ii}) = \tr(DQ) + \tr(Q^* D) \geq 2\tr_{\min}(D),
\]
where $Q$ is the permutation matrix corresponding to the permutation $\sigma$. It follows that $\tr(DQ)=\tr(Q^* D) = \tr_{\min}(D) = 1$ and thus, for each permutation $\sigma$, by the above discussion we must have $a_{i\sigma_{i}}=1/n$ for all $1\leq i \leq n$. Hence, $a_{ij}=1/n$, for every $1\leq i,j \leq n$.
\end{proof}

\subsection{Geometric notions}

\subsubsection*{The Minimal Bounding Ball}

% Given a metric space $(\mathcal{U},d)$, let $\mathcal{B}$ be a nonempty closed, bounded subset of $\mathcal{U}$ and let $B(x,r):= \{ y \in \mathcal{U} : d(x,y) \leq r \}$ denote the closed ball of radius $r>0$ centered at $x\in \mathcal{U}$. We say that $B(x,r)$ is a \emph{bounding ball} of $\mathcal{B}$ centered at $x$ if $\mathcal{B} \subseteq B(x,r)$. The smallest radius $r$ such that $B(x,r)$ is a bounding ball of $\mathcal{B}$ is referred to as the \emph{minimal radius of a bounding ball of $\mathcal{B}$ centered at $x$}. Depending on which information is granted in the context and which is not, it can be either denoted by $r_x(\mathcal{B})$ (for instance see \cite{Goebel1990}) or by $r_d(x)$.

In a metric space $(\mathcal{U},d)$, if $\mathcal{B}\subseteq \mathcal{U}$ is nonempty, closed and bounded, and $B(x,r)$ represents the closed ball with radius $r>0$ centered at $x\in \mathcal{U}$, then we call $B(x,r)$ a \emph{bounding ball} of $\mathcal{B}$ if $\mathcal{B} \subseteq B(x,r)$. The smallest radius $r$ such that $B(x,r)$ is a \emph{bounding ball} of $\mathcal{B}$ is called the \emph{minimal radius of a bounding ball of $\mathcal{B}$ centered at $x$}, denoted by $r_d(x)$.

\begin{figure}[ht]
	\pgfmathsetseed{15946}
	\centering
	\begin{tikzpicture}[scale=1.35]
		\filldraw[fill=lightgray] plot [smooth cycle, samples=5,domain={1:5}] (\x*360/5+5*rnd:0.5cm+1cm*rnd) node at (0,0) {};
		\fill(1,-0.4)node[below,xshift=-3.2cm,yshift=0.3cm]{$\mathcal{B}$} circle (0cm);
		\draw[dashed](-0.4,0.1) circle (1.48);
		\fill(-0.4,0.1)node[below,xshift=-.2cm,yshift=.05cm]{$x$} circle (.04cm);

		\draw(-0.4,0.1)--(1.075,0.01) node[midway,sloped,yshift=-0.28cm]{$r_d(x)$};
	\end{tikzpicture}
	\caption{The smallest enclosing ball of the nonempty closed bounded set $\mathcal{B}$ centered at $x$ with respect to the metric $d$.}
	\label{fig - Psi}
\end{figure}
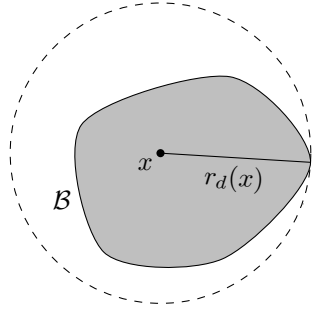  

In this paper, the ambiant space is always $M_n(\mathbb{R})$ and the nonempty, closed and bounded set $\mathcal{B}$ shall always be the Birkhoff polytope $\ds$. Moreover, the metric will also be the one induced by the Schatten $p$-norm. Hence, we have
\begin{align}\label{def: r(x)}
    r_{\|\cdot\|}(A) \,=\, \sup_{B\in \ds} \|A-B\|.
\end{align}
As a matter of fact, the set $\ds$ being compact, the supremum in \eqref{def: r(x)} can be replaced by a maximum, and it follows that the minimal radius of a bounding ball of $\ds$ centered at $A$ exists and is attained for every $A \in M_n(\mathbb{R})$.

\subsubsection*{The Chebyshev radius and center}

The \emph{Smallest Enclosing Ball Problem} is a fundamental geometry question that extends the \emph{Smallest Enclosing Circle Problem} originally proposed by 19th-century mathematician James Joseph Sylvester \cite{sylvester1857question}. Given a metric space $(\mathcal{U},d)$, with a nonempty closed constraint set $\mathcal{R}$ and a nonempty closed, bounded set $\mathcal{B}$, the problem consists in finding the smallest radius $r \geq 0$ and a point $x \in \mathcal{R}$ such that $\mathcal{B}$ is entirely contained within the ball $B(x,r)$, i.e., $\mathcal{B} \subseteq B(x,r)$ and this radius $r$ is the smallest possible. This can be intuitively stated as finding the minimal bounding ball of $\mathcal{B}$ centered at some point in the constraint set $\mathcal{R}$ and relative to the metric space $(\mathcal{U},d)$. 

The minimal bounding ball for $\mathcal{B}$ concerning the metric space $(\mathcal{U},d)$ and the constraint $\mathcal{R}$ may not always exist, and if it does, it might not be unique. When it exists, the radius $r$ is known as the \emph{Chebyshev radius} of $\mathcal{B}$ relative to the metric space $(\mathcal{U},d)$ and the constraint set $\mathcal{R}$, denoted by $R_d(\mathcal{B})$. A point $x \in \mathcal{R}$ that achieves this is called a \emph{Chebyshev center} of $\mathcal{B}$ relative to the metric space $(\mathcal{U},d)$ and the constraint $\mathcal{R}$. Using this notation, observe that we have
$$
R_d(\mathcal{B}) \,=\, \infp_{x\in \mathcal{R}} \sup_{y\in \mathcal{B}} d(x,y) \,=\, \inf_{x\in \mathcal{R}} r_d(x).
$$
When $\mathcal{R}=\mathcal{U}$, the problem is referred to as \emph{unconstrained}. In this setting, we focus on the metric space $(\mathcal{R},d)$, and we define $B(x,r)$ as the minimal bounding ball of $\mathcal{B}$ in relation to the metric space $(\mathcal{R},d)$ (similarly for Chebyshev radius and centers). It's important to note that in these cases, we require $\mathcal{B} \subseteq \mathcal{R}$, unlike the general case where the constraint set $\mathcal{R}$ can be unrelated to $\mathcal{B}$.

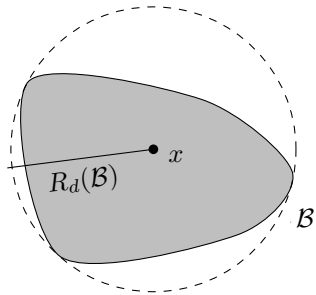
\begin{figure}[ht]
	\pgfmathsetseed{1594627271}
	\centering
	\begin{tikzpicture}[scale=1.6]
		\filldraw[fill=lightgray] plot [smooth cycle, samples=5,domain={1:5}] (\x*360/5+5*rnd:0.5cm+1cm*rnd) node at (0,0) {};
		\fill(1,-0.4)node[below,xshift=.2cm,yshift=.3cm]{$\mathcal{B}$} circle (0cm);
		\draw[dashed](-0.1359,0.205)node[below,xshift=.2cm]{} circle (1.1844);
		\fill(-0.1359,0.205)node[below,xshift=.3cm,yshift=.1cm]{$x$} circle (.04cm);¸
		
		\draw(-0.1359,0.205)--(-1.35,0.05) node[midway,sloped,yshift=-0.3cm]{$R_d(\mathcal{B})$};
		
	\end{tikzpicture}
	\caption{The minimal bounding ball of the nonempty closed, bounded set $\mathcal{B}$ with respect to the metric $d$.}
	\label{fig:centre+rayon}
\end{figure}

More than a century after it has been posed (albeit in a more circumscribed form), the Smallest Enclosing Ball Problem remains an active area of research (see \cite{MR2343160,Mordukhovich2013} and the references therein).
Before beginning our study of the Chebyshev centers and the Chebyshev radius of the Birkhoff polytope relative to the Schatten $p$-norms, we need to recall some known general results. 

First notice that the minimal bounding ball of a set $\mathcal{B}$ does not always exist. 
For instance, in the 2-dimensional Euclidean space with the constraint set $\mathcal{R}=B(0,1)$, where $B(0,1)$ denote the open unit ball, it is clear that the minimal bounding ball of $\mathcal{B}=\{(2,0)\}$ does not exist. Hence, 
we begin by stating some sufficient conditions, due to Mordukhovich, Nguyen Mau and Villalobos \cite{Mordukhovich2013}, guaranteeing the \textit{existence} of a minimal bounding ball of $\mathcal{B}$.

\begin{proposition}\cite[Theorem 1]{Mordukhovich2013}\label{prop - exist}
    Let $\mathcal{B}$ be a nonempty closed bounded subset in the normed vector space $(V,\|\cdot\|)$. Suppose one of the following holds:
    \begin{enumerate}[label=(\roman*)]
        \item The constraint set $\mathcal{R}\subseteq V$ is nonempty and compact.
        \item The normed vector space $(V,\|\cdot\|)$ is a reflexive Banach space and the constraint set $\mathcal{R}\subseteq V$ is weakly closed.
    \end{enumerate}
    Then there exists a minimal bounding ball of $\mathcal{B}$ relative to $(V,\|\cdot\|)$ and the constraint set $\mathcal{R}$.
\end{proposition}

If a minimal bounding ball of $\mathcal{B}$ exists, its uniqueness is not necessarily guaranteed. For example, in the 2-dimensional Euclidean space with $\mathcal{R}=\{(x,y): x^2+y^2=1\}$ and $\mathcal{B}=\{(0,0)\}$, it is clear that a closed unit ball centered at any point of $\mathcal{R}$ is a minimal bounding ball of $\mathcal{B}$. Hence,
we also present a sufficient condition to guarantee the \textit{uniqueness} of the minimal bounding ball of $\mathcal{B}$ in the context of \textit{strictly convex normed vector spaces}, i.e., a normed vector space where $x \neq y$ imply that
\[
\|\lambda x+(1-\lambda)y\| \,<\, \lambda \|x\| + (1-\lambda)\|y\|
\]
for all $0 < \lambda < 1$.

\begin{proposition}\label{prop - uniqueness}
Let $\mathcal{B}$ be a nonempty compact subset in the strictly convex normed vector space $(V,\|\cdot\|)$ and let $\mathcal{R}\subseteq V$ be a nonempty closed, convex constraint set. If there exists a Chebyshev center of $\mathcal{B}$ relative to $(V,\|\cdot\|)$ and the constraint set $\mathcal{R}$, then it is unique.
\end{proposition}
% \begin{proof}
% %Since $\mathcal{B}$ is compact, it is closed and bounded and the restraint set $\mathcal{R}$ is also compact. Therefore, by Proposition \ref{prop - exist}, there exist a minimal bounding ball of $\mathcal{B}$ relative to $(V,\|\cdot\|)$ and the constraint set $\mathcal{R}$ and thus a Chebyshev center.

% Suppose that $x, y \in \mathcal{R}$ are two distinct Chebyshev centers. Then, for every $0 < \lambda < 1$,
%     \begin{align*}
%         R_{\|\cdot\|}(\mathcal{B}) \,&=\, \infp_{w\in \mathcal{R}} \sup_{z\in \mathcal{B}}\|w-z\| \,\leq\, \sup_{z\in \mathcal{B}}\|\lambda x + (1-\lambda)y-z\| \,=\, \|\lambda x + (1-\lambda)y-z_0\| \\
%         &<\, \lambda  \| x-z_0\| + (1-\lambda)  \|y-z_0\| \,\leq\,  \lambda \sup_{z\in \mathcal{B}} \| x-z\| + (1-\lambda) \sup_{z\in \mathcal{B}} \|y-z\| \\[-3pt]
%         %
%         &=\, R_{\|\cdot\|}(\mathcal{B}),
%     \end{align*}
%     where $z_0 \in \mathcal{B}$ exists by compactness. This is a contradiction and thus, the Chebyshev center is unique.
% \end{proof}

Before moving on to the next result, let us recall that the set $\mathcal{K} \subset M_n$ is said to be \textit{permutation-invariant} if $PKQ\in\mathcal{K}$ for any $K\in\mathcal{K}$ and any permutation matrices $P, Q \in \Pn$. Observe that the sets $\ds$ and $M_n(\mathbb{R})$ are both permutation-invariant.

For the remainder of this section, the results will concern the Chebyshev centers and radius of $\ds$ in the context of a metric induced by a permutation-invariant norm and a constraint set which is also permutation-invariant. %In particular, we shall see that under mild conditions, the special matrix $\alpha J_n$ is a Chebyshev center of $\ds$ for some real number $\alpha$ and that the associated Chebyshev radius is given by $\|\alpha J_n-I_n\|$. 

\begin{theorem}\cite[Theorem 6.4]{part1}\label{prop - centre J general}
    Let $\mathcal{R} \subseteq M_n(\mathbb{R})$ be a convex permutation-invariant constraint set and let $\|\cdot\|$ be a permutation-invariant norm on $M_n(\mathbb{R})$. If there exist a Chebyshev center $A$ of $\ds$ relative to the metric space $(M_n(\mathbb{R}),\|\cdot\|)$ and the constraint set $\mathcal{R}$, then the matrix $J_nAJ_n =  \big(\frac{1}{n}\sum_{i,j=1}^n a_{ij} \big)J_n$ is also a Chebyshev center of $\ds$ relative to the aforementioned metric space and constraint set. Moreover, the Chebyshev radius of $\ds$ in this setting is given by
    \[
    R_{\|\cdot \|}(\ds) \,=\,\|J_nAJ_n-I_n\| \,=\, \inf_{\!\!\substack{\alpha\in\mathbb{R}\\ \alpha J_n \in \mathcal{R}}\!}\|\alpha J_n-I_n\|
    \vspace{-2pt}
    \]
    and the infimum is attained by $\alpha = \frac{1}{n}\sum_{i,j=1}^n a_{ij}$.
\end{theorem}

% \begin{proof}
%    Let $A$ be any Chebyshev center in the above settings. Then Lemma \ref{lem - bonus1} 
%     implies that the convex combination $ \frac{1}{(n!)^2} \sum_{P,Q\in\Pn} PAQ$ is also a Chebyshev center. Therefore, an application of Lemma \ref{lem - bonus3} yield 
%     \[
%     \frac{1}{(n!)^2} \!\sum_{P,Q\in\Pn} \!\!PAQ = \left( \frac{1}{n!} \sum_{P\in\Pn} \!P\right)\! A  \left( \frac{1}{n!} \sum_{\smash{Q\in\Pn}} \!Q\right)\! = J_nAJ_n =  \left( \frac{1}{n}\sum_{i,j=1}^n a_{ij} \right)\!J_n.
%     \]
%    Thus, $J_nAJ_n$ is a Chebyshev center of $\ds$ relative to the metric space $ (M_n(\mathbb{R}),\|\cdot\|)$ and the constraint set $\mathcal{R}$.
%     It then follows that $R_{\|\cdot \|}(\ds) = \max_{P\in\Pn} \|J_nAJ_n-P\|$. Consequently, using once again the fact that $\|\cdot\|$ is permutation-invariant and that $J_nP=J_n$ for any permutation matrix $P$, we get
%     \[
%     R_{\|\cdot \|}(\ds) = \max_{P\in \Pn} \|J_nAJ_nP^*-I_n\| = \max_{P\in \Pn} \|J_nAJ_n-I_n\| = \|J_nAJ_n-I_n\|.
%     \]
%        Finally, observe that $J_nAJ_n=\alpha J_n$ for some $\alpha\in\mathbb{R}$. Hence,
%     \begin{align*}
%         R_{\|\cdot \|}(\ds) \,&=\, \inf_{A'\in\mathcal{R}} \max_{P\in\Pn} \|A'-P\| \,\leq \inf_{\substack{\alpha\in\mathbb{R}\\ \alpha J_n \in \mathcal{R}}} \max_{P\in\Pn} \|\alpha J_n-P\| \\
%         %
%         &= \inf_{\substack{\alpha\in\mathbb{R}\\ \alpha J_n \in \mathcal{R}}}\|\alpha J_n-I_n\| \,\leq\, \|J_nAJ_n-I_n\| \,=\, R_{\|\cdot \|}(\ds)
%     \end{align*}
%     and the conclusion follows directly.
% \end{proof}

As a direct consequence of Theorem \ref{prop - centre J general}, we also have the following corollary which states that, under mild conditions, any Chebyshev center of $\ds$ is equidistant to every permutation matrix.

\begin{corollary}\cite[Corollary 6.5]{part1}\label{lem - perm}
Let $\mathcal{R} \subseteq M_n(\mathbb{R})$ be a  convex permutation-invariant constraint set and let $\|\cdot\|$ be a permutation-invariant norm on $M_n(\mathbb{R})$. If there exist a Chebyshev center $A$ of $\ds$ relative to the metric space $(M_n(\mathbb{R}),\|\cdot\|)$ and the constraint set $\mathcal{R}$, then $\|A-P\|=R_{\|\cdot\|}(\ds)$ for any permutation matrix $P\in\Pn$.
\end{corollary}

% \begin{proof}
% Let $A\in\mathcal{R}$ be a Chebyshev center and let $S\in\Pn$ be any permutation matrix. Observe that, regardless of the choice of $S$, Theorem \ref{prop - centre J general} guarantees that
%     \begin{align*}
%         R_{\|\cdot\|}(\ds) = \|J_nAJ_n-I_n\| = \|J_nAJ_n-S\| = \bigg\| \frac{1}{(n!)^2} \!\sum_{\smash{P,Q\in\Pn}} \!\!PAQ - S \bigg\|,
%     \end{align*}
% where the third equality stem from Lemma \ref{lem - bonus3}. Now, Lemma \ref{lem - bonus1} guarantees that $PAQ$ is a Chebyshev center for all $P$ and $Q\in\Pn$ and thus
%     \begin{align*}
%         R_{\|\cdot\|}(\ds) \,&=\, \bigg\| \frac{1}{(n!)^2} \!\sum_{\smash{P,Q\in\Pn}} \!\!PAQ -S \bigg\| \,=\, \frac{1}{(n!)^2} \bigg\| \sum_{\smash{P,Q\in\Pn}} \!(PAQ - S) \bigg\| \\
%         %
%         &\leq\, \frac{1}{(n!)^2} \!\sum_{P,Q\in\Pn} \!\!\left\| PAQ - S \right\| \,\leq\, \frac{1}{(n!)^2} \!\sum_{P,Q\in\Pn} \!\max_{R\in\Pn} \left\| PAQ - R \right\| \\
%         %
%         &=\, \frac{1}{(n!)^2} \!\sum_{P,Q\in\Pn} \!\!R_{\|\cdot\|}(\ds) = R_{\|\cdot\|}(\ds).
%     \end{align*}
% Hence, all the above inequalities are in fact equalities and, in particular, we find that $\left\| PAQ - S \right\| = \max_{R\in\Pn}\left\| PAQ - R \right\| = R_{\|\cdot\|}(\ds)$ for any $P$ and $Q\in\Pn$. Since $\|\cdot\|$ is permutation-invariant, we conclude that $\|A-P^* SQ^*\|=R_{\|\cdot\|}(\ds)$ for every permutation matrices $P\in\Pn$. Hence $\|A-P\|=R_{\|\cdot\|}(\ds)$ for all $P\in\Pn$.
% \end{proof}

The above results are valid even if $\mathcal{R}$ has no direct relation to $\ds$. If we also make the natural assumption that $\mathcal{R}=\ds$, then we obtain the following concrete result.

\begin{corollary}\cite[Corollary 6.6]{part1}\label{cor - centre J}
    If $\|\cdot\|$ is a permutation-invariant norm on $M_n(\mathbb{R})$, then the special doubly stochastic matrix $J_n$ is a Chebyshev center of $\ds$ relative to the metric space $ (\ds,\|\cdot\|)$. Moreover, the associated Chebyshev radius is given by $R_{\|\cdot\|}(\ds)=\|J_n-I_n\|$.
\end{corollary}

\vspace{-7pt}

\section{The minimum and maximum distance of an element of the Birkhoff polytope from the origin}
\label{sec - size}

We begin our discussion by presenting a useful result on the range of the Schatten $p$-norms $\|\cdot\|_{\p}$ ($p \geq 1$) when its entry runs through the Birkhoff polytope $\ds$. Finding lower and upper bounds for this norms (as well as identifying the conditions under which these are optimal) will prove valuable later on.

\begin{proposition}\label{Prop: norm}
	Given $D\in\ds$ and $p \geq 1$, then $1 \leqslant \|D\|_{\p} \leqslant n^{1/p}$ with $\|D\|_{\p}=1$ if and only if $D=J_n$ and $\|D\|_{\p}=n^{1/p}$ if and only if $D$ is a permutation matrix.
\end{proposition}
\begin{proof}
Let $D\in\ds$. The matrix $D^*D$ being real symmetric and positive semi-definite, its eigenvalues are non-negative real numbers. Furthermore, $D^*D$ being doubly stochastic, its eigenvalues are in absolute value smaller or equal to 1. Hence the singular values of $D$ lie in the interval $[0,1]$. Therefore,
\begin{align}\label{eq - ineq}
    1 \,\leq\, \sum_{i=1}^n \sigma_i^p(D) \,\leq\, n,\\[-20pt]\nonumber
\end{align}
and thus $ 1 \leq \|D\|_{\p} \leq n^{1/p}$.

\smallskip
    
Since $1$ is always an eigenvalue of the doubly stochastic matrix $D^*D$, it follows that $\sigma_1(D)=1$. So, if $1=\sum_{i=1}^n \sigma_i^p(D) = \|D\|_{\p}^p$, then $\sigma_j(D)=0$ for $2\leq j \leq n$ and thus $\operatorname{rank}(D)=1$. But M.\ Marcus showed that $J_n$ is the only doubly stochastic matrix of rank 1 \cite[Lemma 2]{Marcus1957}. Consequently $D=J_n$.

Now, if $n=\sum_{i=1}^n \sigma_i^p(D) = \|D\|_{\p}^p$, then  $\sigma_i(D)=1$ for every $1 \leq i \leq n$. Hence, the singular value decomposition of $D$ is given by $D=UI_nV^*=UV^*$, where $U$ and $V$ are unitary matrices. Thus,
$$
D^*D \,=\,(UV^*)^*(UV^*) \,=\,VU^*UV^* \,=\,VV^* \,=\,I_n.
$$
So $D$ is a unitary matrix and its eigenvalues must then all lie on the unit circle. Hence, it follows the first observation made in Section \ref{sec - prop} that $D$ is a permutation matrix. Finally, since $\|P\|_{\p}=n^{1/p}$ for every $n\times n$ permutation matrix, the conclusion follows.
\end{proof}

Note that the above result is not original. Wang, for instance, mentions it in part without proof \cite{Wang1974}.

%%%%%%%%%%%%%%%%%%%%%%%%%%%%%%%%%%%%%%%%%%%%%%%%%%%%%%%%%%%%%%%
%%%%%%%%%%%%%%%%%%%%%%%%%%%%%%%%%%%%%%%%%%%%%%%%%%%%%%%%%%%%%%%
%%%%%%%%%%%%%%%%%%%%%%%%%%%%%%%%%%%%%%%%%%%%%%%%%%%%%%%%%%%%%%%
%%%%%%%%%%%%%%%%%%%%%%%%%%%%%%%%%%%%%%%%%%%%%%%%%%%%%%%%%%%%%%%

\section{The Minimal Bounding Ball of the Birkhoff polytope}
\label{sec - ball}

We now turn our attention to the problem of characterizing the minimal radius of an enclosing ball of $\ds$ centered at $A\in M_n(\mathbb{R})$ when the ambient space is equipped with the Schatten $p$-norms. In some cases, we will make the additional assumption that the matrix $A$ is doubly stochastic; this loss of generality will be compensated by some considerably stronger results.

Finding an explicit formula for $r_{\p}(x)$ for a general $p\geq 1$ turns out to be a difficult problem. We thus restrict our attention to the case of the Frobenius norm (i.e., the Schatten $2$-norm), where the structure arising from the underlying inner product $\langle A,B\rangle_{\p[2]} := \tr(AB^*)$ allows us to derive some interesting formulas. In particular, we are able to characterize the minimal radius of a bounding ball of $\ds$ centered at $A$ in terms of, among other things, the minimal trace $\tr_{\min}(A)$. These formulas will be instrumental in an upcoming paper to better understand the diameter of $\ds$ in the case of the Frobenius norm.

\begin{theorem}\label{thm: rel_rad}
The minimal radius of a bounding ball of $\ds$ centered at $A\in M_n(\mathbb{R})$ relative to the Schatten 2-norm, denoted by  $r_{\p[2]}(A)$, is
$$ r_{\p[2]}(A) \,=\, \big( \|A\|_{\p[2]}^2+n-2\tr_{\min}(A) \big)^{1/2}.$$
\end{theorem}
\begin{proof}
    We know from property \ref{prop - conv-max} of $\ds$ (see Section \ref{sec - prop}) that the maximum of a convex real-valued function on $\ds$ is attained by a permutation matrix. Hence, since $\|A-D\|_{\p[2]}$ is a convex function with respect to the variable $D\in\ds$ for any fixed $A\in M_n(\mathbb{R})$, we have $r_{\p[2]}(A) = \max_{D\in \ds} \|A-D\|_{\p[2]} = \max_{P\in \Pn} \|A-P\|_{\p[2]}$. Moreover, if $P\in\Pn$, Proposition \ref{Prop: norm} ensures us that
    \begin{align*}
        \|A-P\|_{\p[2]}^2 \,&=\, \big\langle A-P,\, A-P \big\rangle_{\p[2]} \\
        &=\, \|A\|_{\p[2]}^2+\|P\|_{\p[2]}^2 - \big\langle A, P \big\rangle_{\p[2]} - \big\langle P, A \big\rangle_{\p[2]} \\
        &=\, \|A\|_{\p[2]}^2+\|P\|_{\p[2]}^2 -  \tr(AP^*) - \tr(PA^*) \\
        &=\, \|A\|_{\p[2]}^2+n - 2 \tr(AP^*).
    \end{align*}
    It follows that
    \[
    r_{\p[2]}^2(A) \,=\, \max_{P\in \Pn} \big(\|A\|_{\p[2]}^2+n - 2 \tr(AP^*)\big)
    \,=\, \|A\|_{\p[2]}^2+n-2\tr_{\min}(A).
    \]
    Taking the square root yield the result we sought.
\end{proof}

It is well known that there exist a connection between the space $\ds$ and the assignment problem  \cite{BURKARD2002257,mehlum2012doubly,Wang1974}. However, to the best of our knowledge, none of these links arose naturally by studying $\ds$. These papers were mostly motivated by the assignment problem and the authors found that by exploiting some of the properties of $\ds$ (i.e., the assignment polytope), they could gain some insight into their original questions. Here, we uncover a new geometric connection between these two objects of study.

However precise it may be, Theorem \ref{thm: rel_rad} is not practical since the minimal trace $\tr_{\min}(A)$ is relatively hard to compute for large $n$. If we make the additional assumption that $A\in\ds$, we can derive a more workable result from the estimations  (i.e., $0\leq \tr_{\min}(A) \leq 1$ for all $A\in\ds$) given in Lemma \ref{borne_m(D)}.

\begin{corollary}
For $D\in\ds$, we have 
\begin{equation}\label{cor - ineq}
    \big( \|D\|_{\p[2]}^2+n-2 \big)^{\frac{1}{2}} \,\leqslant\, r_{\p[2]}(D) \,\leqslant\, \big( \|D\|_{\p[2]}^2+n \big)^{\frac{1}{2}}.
\end{equation}
In particular, the ratio $r_{\p[2]}^2(D)/\!\left(\|D\|_{\p[2]}^2+n\right)$ converges uniformly to $1$ as $n\to\infty$.
\end{corollary}
\begin{proof}
    The lower and upper bounds are obtained directly from Theorem \ref{thm: rel_rad} and Lemma~\ref{borne_m(D)}. For the second part of the proof, note that $r_{\p[2]}^2(D)/\!\left(\|D\|_{\p[2]}^2+n\right)$ converges uniformly to $1$ as $n\to\infty$ if and only if $r_{\p[2]}^2(D)/\!\left(\|D\|_{\p[2]}^2+n-1\right)$ also does. Since the computations in the second one are easier, we instead prove the latter. By squaring \eqref{cor - ineq} and dividing by $\|D\|_{\p[2]}^2+n-1$, we get
    \begin{equation*}
       1 - \frac{1}{\|D\|_{\p[2]}^2+n-1} \,\leq\, \frac{r_{\p[2]}^2(D)}{\|D\|_{\p[2]}^2+n-1} \,\leq\, 1 + \frac{1}{\|D\|_{\p[2]}^2+n-1}.
    \end{equation*}
    That is, 
    \begin{equation*}
        \left|\frac{r_{\p[2]}^2(D)}{\|D\|_{\p[2]}^2+n-1}-1\right| \,\leq\,  \frac{1}{\|D\|_{\p[2]}^2+n-1}.
    \end{equation*}
    Hence, since $\|D\|_{\p[2]}\geq 1$ by Proposition \ref{Prop: norm}, we have 
    \begin{equation*}
        \left|\frac{r_{\p[2]}^2(D)}{\|D\|_{\p[2]}^2+n-1}-1\right| \,\leq\,  \frac{1}{n}
    \end{equation*}
    and the conclusion follows.
\end{proof}

The full force of the inequalities stated will prove instrumental in Section \ref{Sec: Radius}.

Note in closing that, using the estimates for $\|\cdot\|_{\p}$ as $D$ runs through $\ds$ given in Proposition \ref{Prop: norm}, one can deduce upper and lower bounds: 
\[
n-1 \,\leq\, r^2_{\p[2]}(D) \,\leq\, 2n.
\] 
While these bounds are less precise than \eqref{cor - ineq}, they are simpler and somehow more elegant.

%%%%%%%%%%%%%%%%%%%%%%%%%%%%%%%%%%%%%%%%%%%%%%%%%%%%%%%%%%%%%%%
%%%%%%%%%%%%%%%%%%%%%%%%%%%%%%%%%%%%%%%%%%%%%%%%%%%%%%%%%%%%%%%
%%%%%%%%%%%%%%%%%%%%%%%%%%%%%%%%%%%%%%%%%%%%%%%%%%%%%%%%%%%%%%%
%%%%%%%%%%%%%%%%%%%%%%%%%%%%%%%%%%%%%%%%%%%%%%%%%%%%%%%%%%%%%%%

\section{The Chebyshev radius and center of the Birkhoff polytope}
\label{sec - Chebyshev}

\label{Sec: Radius}

We now seek to determine the Chebyshev radius $R_{\p}(\ds)$ and the Chebyshev centers of $\ds$ relative to the metric space $(M_n(\mathbb{R}),\|\cdot\|_{{\p}})$ and $(\ds,\|\cdot\|_{{\p}})$. In fact, we shall see in due course that the latter case follows almost directly from the first.

Observe that $\ds$ is compact and that the metric space $(M_n(\mathbb{R}), \|\cdot\|_{\p})$ is a reflexive Banach space for which $M_n(\mathbb{R})$ is obviously (weakly) closed on itself. Hence, Proposition \ref{prop - exist} guarantees the existence of a Chebyshev center in both of the aforementioned metric spaces. Moreover, since the Schatten $p$-norms are permutation-invariant and since $M_n(\mathbb{R})$ is convex, Theorem \ref{prop - centre J general} tells us that the matrix $\alpha J_n$ is a Chebyshev center for some $\alpha\in\mathbb{R}$ and that the Chebyshev radius is given by $\inf_{\alpha\in\mathbb{R}}\|\alpha J_n-I_n\|_{\p}$. Using these properties, we derive the following result.

\begin{theorem}\label{thm - rad_s1}
For $1< p < \infty$, the special doubly stochastic matrix $J_n$ is the unique Chebyshev center of $\ds$ relative to the metric space $(M_n(\mathbb{R}),\|\cdot\|_{\p})$ and the associated Chebyshev radius is equal to $(n-1)^{1/p}$.
\end{theorem}

\begin{proof}
    We know that, for any $1< p <\infty$, the Chebyshev radius of $\ds$ is given by $\inf_{\alpha\in\mathbb{R}}\|\alpha J_n-I_n\|_{\p}$. Now, observe that $(\alpha J_n-I_n)^*(\alpha J_n-I_n)=I_n-(2\alpha-\alpha^2)J_n$. Moreover, a direct computation shows that the eigenvalues of $J_n$ are 1 and 0. Therefore, the singular values of $\alpha J_n-I_n$ are $1$ with multiplicity $n-1$, and $\sqrt{1-2\alpha+\alpha^2}=|1-\alpha|$ with multiplicity 1. Hence,
    \begin{align}
        \inf_{\alpha\in\mathbb{R}}\|\alpha J_n-I_n\|_{\p} \,&=\, \inf_{\alpha\in\mathbb{R}} \left(\sum_{i=1}^n \sigma_i^p(\alpha J_n-I_n)\!\right)^{\!1/p} \nonumber\\
        &=\, \inf_{\alpha\in\mathbb{R}}\left(n-1 + |1-\alpha|^p\right)^{1/p} \label{eq - schatten-1}\\
        &=\, (n-1)^{1/p}. \nonumber
    \end{align}
      Thus, $R_{\p}(\ds) =(n-1)^{1/p}$ and since the above minimum is attained when $\alpha=1$, the matrix $J_n$ is a Chebyshev center. Since the Schatten $p$-norm are strictly convex for any $1<p<\infty$ \cite{MR4272466},  it follows from Proposition \ref{prop - uniqueness} that the Chebyshev center is unique.
\end{proof}

As a matter of fact, the proof detailed above remains largely valid for $p = 1$. Indeed, only the uniqueness argument fails as the Schatten $1$-norm is not strictly convex. To establish the result for $p=1$, we need stronger tools, including the following lemma.

\begin{lemma}\label{lem - permutation_form}
    Let $A\in M_n(\mathbb{R})$. If $A$ commutes with every $n\times n$ permutation matrices, then there exist some $a,b\in\mathbb{R}$ such that $A=aI_n+bJ_n$. Moreover, if $A\in \ds$, then $a,b\geq 0$ and $a+b=1$.
\end{lemma}
\begin{proof}
    Since $AP=PA$ for every permutation matrices, we have
    \begin{equation*}
        a_{i\sigma(j)} \,=\, (AP)_{ij} \,=\, (PA)_{ij} \,=\, a_{\sigma^{-1}(i)j},  \qquad 1\leq i,j \leq n,
    \end{equation*}
    where $\sigma$ is a permutation associated with the permutation matrix $P$. First consider any permutation matrix $P$ associated with a permutation $\sigma$ satisfying $\sigma(j)=i$. Then we have
    \[
    a_{ii} \,=\, a_{i\sigma(j)} \,=\, a_{\sigma^{-1}(i)j} \,=\, a_{jj}.
    \]
    Since $i$ and $j$ were arbitrary, each entry in the diagonal of $A$ is identical.

    Now consider any permutation matrix $P$ associated with a permutation $\sigma$ satisfying $\sigma(j)=k$ and $\sigma(l)=i$, where $k\neq i$ and $l\neq j$. It is clear that for any $1\leq l,k \leq n$ satisfying $k\neq i$ and $l\neq j$, such a permutation exist. We then have
    \[
    a_{ik} \,=\, a_{i\sigma(j)} \,=\, a_{\sigma^{-1}(i)j} \,=\, a_{lj}.
    \]
    Since $k\neq i$ and $l\neq j$, we find that every off-diagonal entry of $A$ is identical. %It finally follows that there exists some $a,b\in\mathbb{R}$ such that $A=aI_n+bJ_n$. 
    Hence, $A$ has the general form 
    \begin{equation*}
        A \,=\, \begin{bmatrix}
            \alpha&\beta&\cdots&\beta \\
            \beta&\alpha&\cdots&\beta \\
            \vdots&\vdots&\ddots&\vdots\\
            \beta&\beta&\cdots&\alpha
        \end{bmatrix},
    \end{equation*}
    which can be written as $aI_n+bJ_n$, where $b=n\beta$ and $a=\alpha-\beta$. If $A$ is doubly stochastic, then $a,b\geq 0$ since the coefficients are nonnegative and $a+b=1$ since the row and column sums must be equal to $1$.
\end{proof}

Before addressing the following theorem, let us recall von Neumann's trace inequality \cite{vonNeumann1937}, which states that if $B,C\in M_n(\mathbb{R})$, then $|\tr(BC)| \leq \sum_{i=1}^n \sigma_i(B) \sigma_i(C)$, with equality if and only if 
there exist $n\times n$ unitary matrices $U$ and $V$ such that 
\begin{equation}\label{eq - vonneumann}
    B=U\operatorname{diag}(\sigma_1(B),\dots,\sigma_n(B)) V^* \quad\&\quad C = U\operatorname{diag}(\sigma_1(C),\dots,\sigma_n(C)) V^*.
\end{equation}

\begin{theorem}\label{conj2}
    $J_n$ is the unique Chebyshev center of $\ds$ relative to the metric space $(M_n(\mathbb{R}),\|\cdot\|_{\p[1]})$ and the associated Chebyshev radius is equal to $n-1$.
\end{theorem}

\begin{proof}
    The proof that the associated Chebyshev radius is equal to $n-1$ is identical to the proof of Theorem \ref{thm - rad_s1}. Therefore, we just show that $J_n$ is the unique Chebyshev center.

    Let $A\in M_n(\mathbb{R})$ be a Chebyshev center of $\ds$ relative to the metric space $(M_n(\mathbb{R}),\|\cdot\|_{\p[1]})$. By Corollary \ref{lem - perm}, we know that 
    \begin{equation}\label{eq - 6.4}
        \|A-P\|_{\p[1]} \,=\, n-1, \qquad  P\in \Pn.
    \end{equation}
    Moreover, von Neumann's trace inequality shows that $|\tr(B)| \leq \|B\|_{\p[1]}$ for every $B\in M_n(\mathbb{R})$. Hence, we have
    \begin{equation}\label{eq - trace}
        n-1 \,=\, \|A-P^*\|_{\p[1]} \,=\, \|AP-I_n\|_{\p[1]} \,\geq\, |\tr(AP-I_n)| \,=\, |\tr(AP)-n|.
    \end{equation}
    This means in particular that $1-n \leq \tr(AP)-n$, i.e., that $1\leq \tr(AP)$ for every permutation matrix $P$. By summing these inequalities over all permutation matrices, we get
    \begin{equation}\label{eq - trace2}
        1 \,=\, \frac{1}{n!}\sum_{P\in \Pn} \!1 \,\leq\, \frac{1}{n!}\sum_{P\in \Pn} \!\tr(AP) \,=\, \tr\!\Bigg(\!A \,\frac{1}{n!}\!\sum_{P\in \Pn} \!\!P \Bigg) =\, \tr(AJ_n) \,=\, \frac{1}{n} \sum_{i,j=1}^n a_{ij},
    \end{equation}
    where the second to last equality is due to Lemma \ref{lem - bonus3}.

    Now, using once again \eqref{eq - 6.4}, we have in particular that $n-1=\|A-P^*Q^*\|_{\p[1]}$ for every $P,Q\in\Pn$ and it follows from the permutation-invariance of the Schatten norms that $n-1=\|PAQ-I_n\|_{\p[1]}$ for any permutation matrices $P$ and $Q$. Consequently, we find with another application of Lemma \ref{lem - bonus3} that 
    \begin{align*}
        n-1 \,&=\, \frac{1}{(n!)^2} \sum_{P,Q\in \Pn} \!\|PAQ-I_n\|_{\p[1]} \,\geq\, \bigg\| \frac{1}{(n!)^2} \sum_{P,Q\in \Pn} \!(PAQ-I_n) \bigg\|_{\p[1]} \\
        &=\, \bigg\| \frac{1}{(n!)^2} \sum_{P,Q\in \Pn} \!PAQ-I_n \bigg\|_{\p[1]} \!=\, \left\| J_nAJ_n-I_n \right\|_{\p[1]} \,=\, \left\| \alpha J_n-I_n \right\|_{\p[1]} \\
        &=\, n-1 + |1-\alpha|,
    \end{align*}
    where $\alpha:= \frac{1}{n} \sum_{i,j=1}^n a_{ij}$ and the last equality is due to \eqref{eq - schatten-1}. Thus, we have $|1-\alpha|\leq 0$ and it follows that $ \frac{1}{n} \sum_{i,j=1}^n a_{ij} = \alpha = 1$. 
    Therefore, the inequality in \eqref{eq - trace2} is in fact an equality and this happens if and only if
    \begin{equation}\label{eq - trace3}
        \tr(AP) \,=\, 1, \qquad P\in\Pn.
    \end{equation}

    Moreover, it also follows that von Neumann's trace inequality in \eqref{eq - trace} is saturated. Hence, by \eqref{eq - vonneumann} (with $B=AP-I_n$ and $C=I_n$), it follows that there exist unitary matrices $U$ and $V$ such that $AP-I_n=U\operatorname{diag}(\sigma_1(AP-I_n),\dots,\sigma_n(AP-I_n)) V^*$ and $I_n=UV^*$. The latter implies that $U=V$ and thus that $AP-I_n$ is unitarily diagonalizable. %In turn, it follows that $\sigma_i(AP-I_n)=|\lambda_i(AP-I_n)|$ for every $1\leq i \leq n$. Moreover, we have
    % \begin{align*}
    %     n-1 \,&=\, \|AP-I\|_{\p[1]} \,=\, \sum_{i=1}^n \sigma_i(AP-I) \,=\, \sum_{i=1}^n |\lambda_i(AP-I)| \\
    %     %
    %     &\geq\, \left| \sum_{i=1}^n \lambda_i(AP-I) \right| \,=\, |\tr(AP-I)| \,=\, |\tr(AP)-n| \\
    %     %
    %     &=\, n-1,
    % \end{align*}
    % where the last equality is due to \eqref{eq - trace3}. Consequently, every inequality above is saturated and in particular, we find that every eigenvalues of $AP-I$ is real and of the same sign. More precisely, we find that $\sigma_i(AP-I) =-\lambda_i(AP-I)=1-\lambda_i(AP)$ for every $P\in\Pn$.
%
    In particular, this means that $AP$ is also unitarily diagonalizable. Hence, $AP$ is normal for any permutation matrix $P$ and it follows that
    \begin{align*}
        P^*A^*AP \,=\, APP^*A^* \,=\, AA^* \,=\, A^*A.
    \end{align*}
    In other words, the matrix $A^*A$ commutes with $P$ for any permutation matrix $P$. It then follows from Lemma \ref{lem - permutation_form} that $A^*A = aI_n+bJ_n$ for some real numbers $a$ and $b$. Moreover, since $A^*A$ is positive semidefinite, the eigenvalues of $A^*A= aI_n+bJ_n$ are real and nonnegative. Since its eigenvalues are $a+b$ (simple) and $a$ (multiplicity of $n-1$), we must have $a\geq 0$ and $b\geq -a$. 
    
    We know from the unicity of the decomposition of positive definite matrices that if $B$ is a matrix such that $B^*B=aI_n+bJ_n$, then $A$ must be of the form $A=UB$, where $U$ is some unitary matrix. It is then a matter of simple computation to verify that $B=\sqrt{a}I_n+(\sqrt{a + b} - \sqrt{a})J_n=:\alpha I_n +\beta J_n$ is such a matrix and thus that $A=U(\alpha I_n +\beta J_n)$.

    Now, use once again the fact that $A^*A=AA^*$ to obtain 
    \begin{align*}
        aI_n+bJ_n \,&=\, (\alpha I_n +\beta J_n)(\alpha I_n +\beta J_n) \,=\, (\alpha I_n +\beta J_n)U^*U(\alpha I_n +\beta J_n) \,=\,  A^*A \\
        &=\, AA^* \,=\,  U(\alpha I_n +\beta J_n)(\alpha I_n +\beta J_n)U^* \,=\, U(aI_n+bJ_n)U^* \\
        &=\, aI_n + bUJ_n U^*,
    \end{align*}
    and thus that $J_n=UJ_nU^*$. Denote by $r_i$ the $i$th row sum of $U$ and by $c_i$ the $i$th column sum of $U$. From $J_n=UJ_nU^*$, a direct computation reveal that
    \begin{equation*}
        \frac{1}{n} \,=\, (J_n)_{ij} \,=\, (UJ_nU^*)_{ij} \,=\, \frac{r_ir_j}{n}, \qquad 1\leq i,j \leq n.
    \end{equation*}
    In particular, $|r_i|^2=1$ and thus, $r_i=\pm 1$ for any $i\in \{1,2,\dots,n\}$. Suppose that $r_1=1$. Then we also have $\frac{r_j}{n}=\frac{r_1r_j}{n} = \frac{1}{n}$ and thus, $r_j=1$ for every $j\in \{1,2,\dots,n\}$. Similarly, if $r_1=-1$ then $r_j=-1$ for every $j\in \{1,2,\dots,n\}$.

    Moreover, from $J_n=UJ_nU^*$ we also find that $J_nU=UJ_n$ which is equivalent to having
    \begin{align*}
        \frac{r_j}{n} \,=\, (J_nU)_{ij} \,=\, (UJ_n)_{ij} \,=\, \frac{c_i}{n}
    \end{align*}
    and thus, $r_j=c_i$ for every $i,j\in \{1,2,\dots,n\}$. Consequently, every row and column of $U$ sum to either $1$ or $-1$. Suppose that it sum to $1$. Then we know from \cite{Lin2017} that $U$ must be of the form $U=\sum_{m} c_m P_m$, where $P_m$ are permutation matrices and $c_m$ are complex numbers such that $\sum_m c_m=1$ and $\sum_m |c_m|^2=1$. If the rows and columns of $U$ sum to $-1$, then $U$ is the form $U=-\sum_{m} c_m P_m$.

    From this, it follows that
    \[
    A \,=\, U(\alpha I_n +\beta J_n) \,=\, \alpha U \pm \beta J_n,
    \]
    since $UJ_n= \pm \sum_{m} c_m P_mJ_n = \pm\sum_{m} c_m J_n = \pm J_n$. Now, recall that $\frac{1}{n} \sum_{i,j=1}^n a_{ij}=1$ and thus
    \begin{align*}
        1 \,=\, \frac{1}{n} \sum_{i,j=1}^n a_{ij} \,=\, \frac{1}{n} \alpha \sum_{i,j=1}^n u_{ij} \pm \frac{1}{n} \beta \sum_{i,j=1}^n \frac{1}{n} \,=\, \pm(\alpha + \beta ),
    \end{align*}
    where the last equality stem from the fact that each row of $U$ sum to $\pm1$. Since $\alpha=\sqrt{a}\geq 0$ and $\beta=\sqrt{a + b} - \sqrt{a} \geq 0$, it follows that $\pm=+$, i.e., that $U=\sum_{m} c_m P_m$ and thus that $\alpha+\beta=1$.

    Finally, recall \eqref{eq - trace3} which states that $\tr(AP)=1$ for every permutation matrix $P$. By the linearity of the trace and the fact that $\tr(J_nP)=\tr(J_n)=cst$, we find that $\tr(UP)$ is also constant for every permutation matrix $P$, say $\tr(UP)=k$. Therefore, since $U$ is unitary we have
    \begin{align*}
        n \,=\, \tr(I_n) \,=\, \tr(UU^*) \,=\, \tr\!\bigg(U \sum_{m} \overline{c_m} P_m^* \bigg) \,=\,  \sum_{m} \overline{c_m} \tr(U P_m^*) \,=\, k\sum_{m} \overline{c_m} \,=\, k.
    \end{align*}
    Therefore, we must have $\tr(UP)=k=n$ for every permutation matrix $P$. 
    However, we also have
    \begin{align*}
        1\,&=\, \tr(AP) \,=\, \tr((\alpha U + \beta J_n)P) \,=\, \alpha\tr(UP) + \beta \tr(J_n) \,=\,  \alpha\tr(UP) + \beta \\
        &=\, \alpha\tr(UP) + 1-\alpha,
    \end{align*}
    which means that $\alpha = \alpha \tr(UP)$ for every permutation matrix $P$. If $\alpha \neq 0$, then this implies that $1=\tr(UP)=n$, a contradiction (since we can suppose without any loss of generality that $n>1$). Therefore, $\alpha$ must be equal to $0$ and consequently, we find that
    \[
    A \,=\, \beta J_n \,=\, J_n,
    \]
    which is what we wanted to show.
\end{proof}

From the previous theorems, the particular case where the constraint set is limited to $\ds$ follows almost immediately since the Chebyshev center $J$ belong to $\ds \subseteq M_n(\mathbb{R})$.

\begin{corollary}\label{cor - unique_Schatten}
For $1\leq p < \infty$, the special doubly stochastic matrix $J_n$ is the unique Chebyshev center of $\ds$ relative to the metric space $(\ds,\|\cdot\|_{\p})$ and the associated Chebyshev radius is equal to $(n-1)^{1/p}$.
\end{corollary}
\begin{proof}
   Fix $1\leq p < \infty$ and suppose that $D_0 \in M_n(\mathbb{R})$ is a Chebyshev center of $\ds$ relative to the metric space $(M_n(\mathbb{R}),\|\cdot\|_{\p})$. If $D_0 \in \ds$, then we have
   \begin{align*}
        (n-1)^{\frac{1}{p}} &= \!\infp_{A\in M_n(\mathbb{R})} \sup_{P\in \Pn} \|A-P\|_{\p} \leq \infp_{D\in \ds}\sup_{P\in \Pn} \|D-P\|_{\p} \\
        &\leq \sup_{P\in \Pn} \|D_0-P\|_{\p} = (n-1)^{\frac{1}{p}}.
    \end{align*}
    Therefore, $R_{\p}(\ds)=(n-1)^{1/p}$. By Theorems \ref{thm - rad_s1} and \ref{conj2}, $J_n\in\ds$ is the only $n\times n$ real matrix satisfying $\sup_{P\in \Pn} \|D-P\|_{\p} \,=\, (n-1)^{1/p}$ and in particular, it is the only doubly stochastic matrix satisfying this equation. Consequently, $J_n$ is the unique Chebyshev center of $\ds$ relative to the metric space $(\ds,\|\cdot\|_{\p})$.
\end{proof}

\begin{remark}
    More generally, suppose that $\mathcal{B}$ be a nonempty closed bounded subset in the metric space $(V,\|\cdot\|)$, and let $\mathcal{R}_1,\mathcal{R}_2\subseteq V$ be two nonempty closed constraint set. If $\mathcal{R}_1\subseteq \mathcal{R}_2$ and $A$ is the unique Chebyshev center of $\mathcal{B}$ relative to the metric space $(V,\|\cdot\|)$ and the constraint set $\mathcal{R}_2$, then $A$ is also the unique Chebyshev center of $\mathcal{B}$ relative to the metric space $(V,\|\cdot\|)$ and the constraint set $\mathcal{R}_1$. Moreover, the Chebyshev radius is equal in both settings.
\end{remark}

\section{Concluding Remarks}

To conclude this paper, we state a few remarks and set out some open questions that appear to be of interest.

\begin{enumerate}[label=(\roman*)]
    \item Can we find formulas for the minimal radius of a bounding ball of $\ds$ centered at $D\in\ds$ relative to the Schatten $p$-norm for $p\neq2$, and in particular for $p=1$ and $p=\infty$?

    \item The problem of determining $\tr_{\min}(D)$ is an important one in the area of combinatorial optimization. Is it possible to find an efficient algorithm to compute $\tr_{\min}(D)$ for $D\in\ds$ by using the fact that $\tr_{\min}(D) = (\|D\|_{\p[2]}^2+n-r_{\p[2]}^2(D))/2$?

   % \item {\color{red}Is Conjecture \ref{conj2} valid?}

    %\item While the matrix $J_n$ is the Chebyshev center of $\ds$, it is also the \emph{barycenter} of the set of permutation matrices, i.e., it is the arithmetic mean of all the $n\times n$ permutation matrices (see Lemma \ref{lem - bonus3}). Under which conditions the Chebyshev center of the convex hull of a finite set of matrices $\mathcal{M}$ coincide with the barycenter of the set $\mathcal{M}$?

    %\item Is Conjecture \ref{conj} valid for every $n\geq4$?
\end{enumerate}

\bibliographystyle{plain}
\bibliography{Part2_Schatten}
%\nocite{*}

\end{document}